\RequirePackage{fix-cm}
\documentclass[smallextended]{svjour3}       % onecolumn (second format)
\smartqed  % flush right qed marks, e.g. at end of proof
\usepackage{graphicx}
\usepackage{latexsym,amsmath,enumerate,amssymb,amsbsy}  %by Fuchun Lin
\usepackage{url}                                                                          %by Fuchun Lin
\begin{document}

\title{$2$- and $3$-modular Lattice Wiretap Codes in Small Dimensions%\thanks{Grants or other notes
%about the article that should go on the front page should be
%placed here. General acknowledgments should be placed at the end of the article.}
}
%\subtitle{Do you have a subtitle?\\ If so, write it here}

%\titlerunning{Short form of title}        % if too long for running head

\author{Fuchun~Lin        \and
        Fr\'ed\'erique~Oggier \and
        Patrick~Sol\'e%etc.
}

%\authorrunning{Short form of author list} % if too long for running head

\institute{Fuchun Lin and Fr\'ed\'erique Oggier \at
              Division of Mathematical Sciences,
  School of Physical and Mathematical Sciences, Nanyang Technological
  University, 21 Nanyang Link, Singapore 637371 \\
              %Tel.: +123-45-678910\\
              %Fax: +123-45-678910\\
              \email{ linf0007@e.ntu.edu.sg and frederique@ntu.edu.sg}           %  \\
%             \emph{Present address:} of F. Author  %  if needed
           \and
           Patrick Sol\'e \at
              Telecom ParisTech, CNRS, UMR 5141, Dept Comelec, 46 rue Barrault 75634 Paris cedex 13, France and Mathematics
Department, King AbdulAziz Unversity, Jeddah, Saudi Arabia.\\ 
               \email{ patrick.sole@telecom-paristech.fr}
}

\date{Received: date / Accepted: date}
% The correct dates will be entered by the editor

\maketitle

\begin{abstract}
A recent line of work on lattice codes for Gaussian wiretap channels introduced a new lattice invariant called secrecy gain as a code design criterion which captures the confusion that lattice coding produces at an eavesdropper. Following up the study of unimodular lattice wiretap codes \cite{preprint}, this paper investigates $2$- and $3$-modular lattices and compares them with unimodular lattices. Most even $2$- and $3$-modular lattices are found to have better performance (that is, a higher secrecy gain) than the best unimodular lattices in dimension $n,\ 2\leq n\leq 23$. Odd $2$-modular lattices are considered, too, and three lattices are found to outperform the best unimodular lattices.
\keywords{Wiretap codes \and Gaussian channel \and Lattice codes \and Secrecy gain \and  Modular lattices\and Theta series }
% \PACS{PACS code1 \and PACS code2 \and more}
% \subclass{MSC code1 \and MSC code2 \and more}
\end{abstract}

\section{Introduction}
\label{sec:introduction}
In his seminal work, Wyner introduced the wiretap channel \cite{Wyner}, a discrete memoryless channel where the sender Alice
transmits confidential messages to a legitimate receiver Bob, in the presence of an eavesdropper Eve, who has only partial access to what Bob sees.
Both reliable and confidential communication between Alice and Bob is shown to be achievable at the same time,
by exploiting the physical difference between the channel to Bob and that to Eve, without the use of cryptographic means. Since then, many results of information theoretical nature have been found for various classes of wiretap channels ranging from Gaussian point-to-point channels to relay networks (see e.g. \cite{theoretic security} for a survey) capturing the trade-off between reliability and secrecy and aiming at determining the highest information rate that can be achieved with perfect secrecy, the so-called \textit{secrecy capacity}. Coding results focusing on constructing concrete codes that can be implemented in a specific channel are much fewer (see \cite{OW-84,TDCMM-07} for wiretap codes dealing with channels with erasures, \cite{Polar} for Polar wiretap codes and \cite{wiretap Rayleigh fading} for wiretap Rayleigh fading channels.

In this paper, we will focus on Gaussian wiretap channels, whose secrecy capacity was established in \cite{IEEE Gaussian wiretap channel}. Examples of existing Gaussian wiretap codes were designed for binary inputs, as in \cite{LDPC Gaussian wiretap codes,nested codes}. A different approach was adopted in \cite{ISITA}, where lattice codes were proposed, using as design
criterion a new lattice invariant called \textit{secrecy gain}, defined as the maximum of its \textit{secrecy function} (Section II), which was shown to characterize the confusion at the eavesdropper. A recent study on a new design criterion called \textit{flatness factor} confirms that to confuse Eve, the secrecy gain should be maximized \cite{flatness factor}. This suggests the study of the secrecy gain of lattices as a way to understand how to design a good Gaussian lattice wiretap code. Belfiore and Sol\'{e} \cite{ITW} discovered a symmetry point, called \textit{weak secrecy gain}, in the secrecy function of \textit{unimodular} lattices (generalized to all \textit{$\ell$-modular} lattices \cite{Gaussian wiretap codes}) and conjectured that the weak secrecy gain is actually the secrecy gain. Anne-Maria Ernvall-Hyt\"{o}nen \cite{EH,EHKissingnumber} invented a method to prove or disprove the conjecture for unimodular lattices. Up to date, secrecy gains of a special class of unimodular lattices called \textit{extremal unimodular} lattices and all unimodular lattices in dimensions up to $23$ are computed  \cite{Gaussian wiretap codes,preprint}. The asymptotic behavior of the average weak secrecy gain as a function of the dimension $n$ was investigated and an achievable lower bound on the secrecy gain of even unimodular lattices was given \cite{Gaussian wiretap codes}. Numerical upper bounds on the secrecy gains of unimodular lattices in general and unimodular lattices constructed from self-dual binary codes were given to compared with the achievable lower bound \cite{ITW2012}. 

This paper studies the weak secrecy gain of $2$- and $3$-modular lattices. Preliminary work \cite{ISITversion} showed that most of the known even $2$- and $3$-modular lattices in dimensions up to $24$ have secrecy gains bigger than the best unimodular lattices. After recalling how to compute the weak secrecy gain of even $2$- and $3$-modular lattices using the theory of modular forms, we extend our study to a class of odd $2$-modular lattices constructed from self-dual codes. We propose two methods to compute their weak secrecy gains and find three of these lattices have secrecy gains bigger than the best unimodular lattices. We then conclude that, at least in dimensions up to $23$, $2$- and $3$-modular lattices are a better option than unimodular lattices.
%\item We find the corresponding self-dual codes for the good $2$- and $3$-modular lattices and give an example to illustrate the encoding via Construction A.
%\end{itemize} 

The remainder of this paper is organized as follows. In Section \ref{sec:preliminaries}, we first give a brief introduction to modular lattices and their \textit{theta series} as well as recall the definition of the secrecy gain and the previous results concerning this lattice invariant. The main results are given in Section \ref{sec:results}. Two approaches to compute the theta series of modular lattices are given, one making use of the modular form theory while the other utilizing the connection between the theta series and the weight enumerator of self-dual codes. Weak secrecy gains of several $2$- and $3$-modular lattices computed are then compared with the best unimodular lattices in Section \ref{sec:comparision}. In Section \ref{sec:conclu}, we summarize our results and give some future works.

%%%%%%%%%%%%%%%%%%%%%%%%%%%%%%%%%%%%%%%%%%%%%%%%%%%%%%%%%%%%%%%%%%%%%%%%%%%%%%%%%%%%%%%%%%
%%%%%%%%%%%%%%%%%%%%%%%%%%%%%%%%%%%%%%%%%%%%%%%%%%%%%%%%%%%%%%%%%%%%%%%%%%%%%%%%%%%%%%%%%%
%%%%%%%%%%%%%%%%%%%%%%%%%%%%%%%%%%%%%%%%%%%%%%%%%%%%%%%%%%%%%%%%%%%%%%%%%%%%%%%%%%%%%%%%%%
\section{Preliminaries and previous results}\label{sec:preliminaries}
Consider a Gaussian wiretap channel, which is modeled as follows: Alice wants to send data to Bob over a Gaussian channel whose noise variance is given by $\sigma_b^2$. Eve is the eavesdropper trying to intercept data through another Gaussian channel with noise variance $\sigma_e^2$, where $\sigma_b^2< \sigma_e^2$, in order to have a positive secrecy capacity \cite{IEEE Gaussian wiretap channel}. More precisely, the model is
\begin{equation}\label{channel model}
\begin{array}{cc}
\mathbf{y}&=\mathbf{x}+\mathbf{v_b}\\
\mathbf{z}&=\mathbf{x}+\mathbf{v_e}.\\
\end{array}
\end{equation}
$\mathbf{x}\in \mathbb{R}^n$ is the transmitted signal. $\mathbf{y}$ and $\mathbf{z}$ are the received signals at Bob's, respectively Eve's side. $\mathbf{v_b}$ and $\mathbf{v_e}$ denote the Gaussian noise vectors at Bob's, respectively Eve's side, each component of both vectors are with zero mean, and respective variance $\sigma_b^2$ and $\sigma_e^2$. In this paper, we choose $\mathbf{x}$ to be a codeword coming from a specially designed lattice of dimension $n$, namely, we consider lattice coding. Let us thus start by recalling some concepts concerning lattices, in particular, \textit{modular lattices}.

A \textit{lattice} $\Lambda$ is an additive subgroup of $\mathbb{R}^n$, which can be described in terms of its \textit{generator matrix} $M$ by 
$$
\Lambda=\{\mathbf{x}=\mathbf{u}M|\mathbf{u}\in \mathbb{Z}^m\}, 
$$  
where
$$
M=\left (
\begin{array}{cccc}
v_{11}&v_{12}&\cdots&v_{1n}\\
v_{21}&v_{22}&\cdots&v_{2n}\\
\cdots&&\cdots&\\
v_{m1}&v_{m2}&\cdots&v_{mn}\\
\end{array}
\right )
$$
and the row vectors $\mathbf{v}_i=(v_{i1},\cdots,v_{in}),\ i=1,\ 2,\ \cdots,\ m$ form a basis of the lattice $\Lambda$. The matrix
$$
G=MM^T,
$$
where $M^T$ denotes the transpose of $M$, is called the \textit{Gram matrix} of the lattice. It is easy to see that the $(i,j)$th entry of $G$ is the inner product of the $i$th and $j$th row vectors of $M$, denoted by
$$G_{(i,j)}=\mathbf{v}_i\cdot \mathbf{v}_j.$$ 
%A lattice $\Lambda$ is called an \textit{integral lattice} if the entries of its Gram matrix are all integers. 
The \textit{determinant} $\mbox{det}(\Lambda)$ of a lattice $\Lambda$ is the determinant of the matrix $G$, which is independent of the choice of the matrix $M$. A \textit{fundamental region} for a lattice is a building block which when repeated many times fills the whole space with just one lattice point in each copy. There are many different ways of choosing a fundamental region for a lattice $\Lambda$, but the volume of the fundamental region is uniquely determined and called the \textit{volume} $\mbox{vol}(\Lambda)$ of $\Lambda$, which is exactly $\sqrt{\mbox{det}(\Lambda)}$. Let us see an example of a fundamental region of a lattice. A \textit{Voronoi cell} $\mathcal{V}_{\Lambda}(\mathbf{x})$ of a lattice point $\mathbf{x}$ in $\Lambda$ consists of the points in the space that are closer to $\mathbf{x}$ than to any other lattice points of $\Lambda$.

The \textit{dual} of a lattice $\Lambda$ of dimension $n$ is defined to be
$$
\Lambda^*=\{\mathbf{x}\in \mathbb{R}^n: \mathbf{x}\cdot \mathbf{\lambda}\in \mathbb{Z} , \mbox{ for all } \mathbf{\lambda}\in \Lambda\}.
$$
A lattice $\Lambda$ is called an \textit{integral lattice} if $\Lambda\subset \Lambda^*$. The norm of any lattice point in an integral lattice $\Lambda$ is always an integer. If the norm is even for any lattice point, then $\Lambda$ is called an \textit{even} lattice. Otherwise, it is called an \textit{odd} lattice. A lattice is said to be \textit{equivalent}, or geometrically similar to its dual, if it differs from its dual only by possibly a rotation, reflection and change of scale. An integral lattice that is equivalent to its dual is called a \textit{modular} lattice. Alternatively as it was first defined by H.-G. Quebbemann \cite{modular lattices},  an $n$-dimensional integral lattice $\Lambda$ is modular if there exists a similarity $\sigma$ of $\mathbb{R}^n$ such that $\sigma(\Lambda^{*})=\Lambda$. If $\sigma$ multiplies norms by $\ell$, $\Lambda$ is said to be $\ell$-\textit{modular}. The determinant of an $\ell$-modular lattice $\Lambda$ of dimension $n$ is given by 
\begin{equation}\label{determinant}
\mbox{det}(\Lambda)=\ell^{\frac{n}{2}}. 
\end{equation}
This is because, on the one hand, $\mbox{det}(\Lambda^{*})=\mbox{det}(\Lambda)^{-1}$ by definition and, on the other hand, $\ell^n\mbox{det}(\Lambda^{*})=\mbox{det}(\Lambda)$ since $\sigma(\Lambda^{*})=\Lambda$. When $\ell=1$, $\mbox{det}(\Lambda)=1$ and we recover the definition of unimodular lattice as an integral lattice whose determinant is $1$.
% In the particular case when the similarity factor $\ell=1$, such lattices are \textit{unimodular} lattices, corresponding to the more familiar definition $\Lambda=\Lambda^*$.
%Especially, if $\Lambda=\Lambda^*$ then $\Lambda$ is called a \textit{unimodular} lattice. It can further be shown that $\Lambda$ is a unimodular lattice if and only if $\Lambda$ is integral and $\mbox{det}\Lambda=1$.

\begin{example}
\begin{equation}\label{Z+lZ}
C^{\ell}=\sum_{d|\ell}\sqrt{d}\mathbb{Z},\ \ell=1,2,3,5,6,7,11,14,15,23
\end{equation}
is an $\ell$-modular lattice \cite{N.J.A Sloane}. When $\ell$ is a prime number, $C^{\ell}=\mathbb{Z}\oplus\sqrt{\ell}\mathbb{Z}$ is a two-dimensional $\ell$-modular lattice with the similarity map $\sigma$ taking $(x,y)$ to $(\sqrt{\ell}y,\sqrt{\ell}x)$. 
\end{example}

%*****************************************************************************************%
% CONSTRUCTION A                                                                                                      %
%*****************************************************************************************%

We will use some terminology from classical error correction codes in this paper. Unfamiliar readers can refer to \cite{error correction codes}. We will also assume basic knowledge of algebraic number theory \cite{algebraic number theory}. There is a classical way of constructing $\ell$-modular lattices from self-dual codes called Construction A. Let $K=\mathbb{Q}(\sqrt{\mu})$ be a quadratic imaginary extension of the rational field $\mathbb{Q}$ constructed by adjoining to it the square root of a square free negative integer $\mu$. The ring of integers $\mathfrak{O}_K$ of $K$ is given by  
\begin{equation}\label{eq: ringofintegers}
\mathfrak{O}_K=\mathbb{Z}[\theta],\ \theta=\left\{
\begin{array}{ll}
\frac{1+\sqrt{\mu}}{2}, &\mu\equiv 1\  (\mbox{mod }4) \\
\sqrt{\mu}, &\mbox{otherwise.} \\
\end{array}
\right.
\end{equation}
Let $p$ be a prime number. Then the quotient ring $R=\mathfrak{O}_K/p\mathfrak{O}_K$ is given by
\begin{equation}\label{eq: quotientringR}
R=\left\{
\begin{array}{ll}
\mathbb{F}_p\times\mathbb{F}_p, &p \mbox{ is split in }K; \\
\mathbb{F}_p+u\mathbb{F}_p\mbox{ with }u^2=0, &p\mbox{ is ramified in }K; \\
\mathbb{F}_{p^2}, &p\mbox{ is inert in }K. \\
\end{array}
\right.
\end{equation}
Let $k$ be a positive integer. Let
$$\rho: \mathfrak{O}_K^k\rightarrow R^k$$ 
be the map of component wise reduction modulo $p\mathfrak{O}_K$. 
Then the pre-image $\rho^{-1}(C)$ of a self-dual code $C$ over $R$ of length $k$ with carefully chosen $\mu$ and $p$ and possibly a re-scaling can give rise to a real $\ell$-modular lattice of dimension $2k$ \cite{CBachoc,PSole}. %This construction is a generalization to \textit{complex lattices} of the construction used in \cite{preprint}, though complex lattices will not be discussed in this paper. 
Examples will be specified in the sequel.

%\begin{example}\label{ideallatticeO}
%Let $K=\mathbb{Q}(\sqrt{-2})$ be the quadratic extension of $\mathbb{Q}$ obtained by adjoining the complex number $\sqrt{-2}$ to $\mathbb{Q}$. The ring of integers $\mathfrak{O}_K$ of $K$ is then $\mathbb{Z}[\sqrt{-2}]=\{a+b\sqrt{-2}|a,b\in\mathbb{Z}\}$, since $-2\equiv 2\mbox{ mod }4$. Now since $3=(1+\sqrt{-2})(1-\sqrt{-2})$ and $(1+\sqrt{-2})\mathfrak{O}_K\neq(1-\sqrt{-2})\mathfrak{O}_K$, namely, the ideal $3\mathfrak{O}_K$ splits, the quotient group $\mathfrak{O}_K/3\mathfrak{O}_K$ is the ring $\mathbb{F}_3\times \mathbb{F}_3$, which is equivalent to $R=\mathbb{F}_3+v\mathbb{F}_3$. Now we can see that $v$ is actually $\sqrt{-2}+3\mathfrak{O}_K$.
%\end{example}

\begin{definition}
The theta series of a lattice $\Lambda$ is defined by
$$\Theta_{\Lambda}(\tau)=\Sigma_{\mathbf{\lambda}\in \Lambda}q^{\mathbf{||\lambda}||^2},q=e^{\pi i \tau},\ \tau\in \mathcal{H},$$
where $\mathbf{||\lambda}||^2=\mathbf{\lambda}\cdot\mathbf{\lambda}$ is called the (squared) norm of $\mathbf{\lambda}$ and $\mathcal{H}=\{a+ib\in \mathbb{C}|b>0\}$ denotes the upper half plane.
\end{definition}
%The theta series of a lattice is in general difficult to analyze. They are actually \textit{modular forms} \cite{modular forms}. 

The theta series of an integral lattice has a neat representation. Since the norms are all integers, we can combine the terms with the same norm and write
\begin{equation}\label{equ:theta series}
\Theta_{\Lambda}(\tau)=\Sigma_{m=0}^{\infty}A_mq^m,
\end{equation}
where $A_m$ counts the number of lattice points with norm $m$. They are actually \textit{modular forms} \cite{modular forms}. 

%**************************************************
%Jacobi theta functions
We will also need the following functions and formulae from analytic number theory for our discussion, for which interested readers can refer to \cite{analytic number theory}.
\begin{definition}
The Jacobi theta functions are defined as follows: 
$$
\left\{
\begin{array}{ll}
\vartheta_2(\tau)&=\Sigma_{m\in\mathbb{Z}}q^{(m+\frac{1}{2})^2},\\
\vartheta_3(\tau)&=\Sigma_{m\in\mathbb{Z}}q^{m^2},\\
\vartheta_4(\tau)&=\Sigma_{m\in\mathbb{Z}}(-q)^{m^2}.\\
\end{array}
\right.
$$
\end{definition}
\begin{definition} The Dedekind eta function is defined by
$$
\eta(\tau)=q^{\frac{1}{12}}\prod^\infty_{m=1}(1-q^{2m}).
$$
\end{definition}

The Jacobi theta functions and the Dedekind eta function are connected as follows \cite{analytic number theory}:
\begin{equation}\label{theta and eta}
\left\{
\begin{array}{ll}
\vartheta_2(\tau)&=\frac{2\eta(2\tau)^2}{\eta(\tau)},\\
\vartheta_3(\tau)&=\frac{\eta(\tau)^5}{\eta(\frac{\tau}{2})^2\eta(2\tau)^2},\\
\vartheta_4(\tau)&=\frac{\eta(\frac{\tau}{2})^2}{\eta(\tau)}.\\
\end{array}
\right.
\end{equation}
%=q^{\frac{1}{4}}\prod^\infty_{n=1}(1-q^{2n})(1+q^{2n})(1+q^{2n-2})
%=\prod^\infty_{n=1}(1-q^{2n})(1+q^{2n-1})^2
%=\prod^\infty_{n=1}(1-q^{2n})(1-q^{2n-1})^2

%The \textit{trace function} is defined by 
%$$
%\mbox{Tr}(x)=x+\bar{x},\ x\in K,
%$$
%where $\bar{x}$ is the complex conjugation.

%*****************************************************************************************%
% Previous results                                                                                                          %
%*****************************************************************************************%

%Secrecy gain

Lattice encoding for the wiretap channel (\ref{channel model}) is done via a generic coset coding strategy \cite{ISITA}: let $\Lambda_e\subset\Lambda_b$ be two nested lattices. A $k$-bit message is mapped to a coset in $\Lambda_b/\Lambda_e$, after which a vector is randomly chosen from the coset as the encoded word. The lattice $\Lambda_e$ can be interpreted as introducing confusion for Eve, while $\Lambda_b$ is intended to ensure reliability for Bob. Since a message is now corresponding to a coset of codewords instead of one single codeword, the probability of correct decoding is then summing over the whole coset (suppose that we do not have power constraint and are utilizing the whole lattice to do the encoding). Here we are interested in computing $P_{c,e}$, Eve's probability of correct decision, and want to minimize this probability. It was shown in \cite{ISITA,Gaussian wiretap codes} that to minimize $P_{c,e}$ is to minimize 
\begin{equation}\label{secrecy theta series}
\sum_{\mathbf{t}\in \Lambda_e}e^{-||\mathbf{t}||^2/2\sigma_e^2},
\end{equation}
which is easily recognized as the theta series of $\Lambda_e$ at $\tau=\frac{i}{2\pi\sigma_e^2}$. We hence only care about values of $\tau$ such that $\tau=yi,\ y>0$.

Motivated by the above argument, the confusion brought by the lattice $\Lambda_e$ with respect to no coding (namely, use a scaled version of the lattice $\mathbb{Z}^n$ with the same volume) is measured as follows:
\begin{definition} \cite{ISITA}
Let $\Lambda$ be an $n$-dimensional lattice of volume $v^n$. The secrecy function of $\Lambda$ is given by
$$\Xi_{\Lambda}(\tau)=\frac{\Theta_{v\mathbb{Z}^n}(\tau)}{\Theta_{\Lambda}(\tau)}, \tau=yi, y>0 .$$
The \textit{secrecy gain} is then the maximal value of the secrecy function with respect to $\tau$ and is denoted by $\chi_{\Lambda}$.
\end{definition}

\begin{figure}[htp]

\centering
\includegraphics[width=80mm, height=60mm]{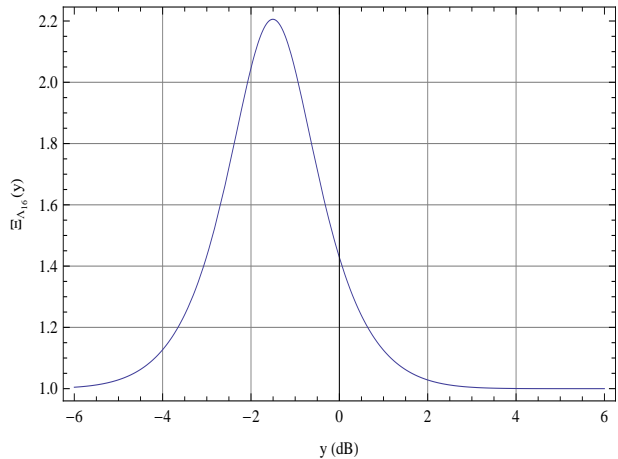}
\caption{\label{fig:secrecy function of BW_16}Secrecy function of $BW_{16}$}
\end{figure}
 
$\ell$-modular lattices were shown to have a symmetry point, called \textit{weak secrecy gain} $\chi^w_{\Lambda}$, at $\tau=\frac{i}{\sqrt{\ell}}$ in their secrecy function \cite{Gaussian wiretap codes}. See  Fig. \ref{fig:secrecy function of BW_16} for an example, where $y$ is plotted in dB to transform the multiplicative symmetry point into an additive symmetry point. $BW_{16}$ is a $2$-modular lattice. One can see there is a symmetry point at $y=-\frac{3}{2}$ dB, which is $\frac{\sqrt{2}}{2}$. This paper is devoted to computing the weak secrecy gain of $2$- and $3$-modular lattices in small dimensions.
%We will still denote the weak secrecy gain by $\chi^w_{\Lambda}$ since all the weak secrecy gains computed in this paper are also shown to be the secrecy gains. This class of lattices contains lattices whose duals are obtained from themselves by possibly a rotation, reflection, and change of scale. Let us now focus on unimodular lattices, for which we have $\Lambda^{*}=\Lambda$ by definition. It was a conjecture by Belfiore and Sol\'e \cite{ITW}, that for unimodular lattices, $\tau=i$ is not only the symmetry point, but also the point achieving the secrecy gain:

%**************************************************************************************************************************************%
%                                                                                                                                                                                                         %
%SEC III The theta series of 2- and 3-modular lattices                                                                                                                                %
%                                                                                                                                                                                                         %
%**************************************************************************************************************************************%
\section{The weak secrecy gain of $2$- and $3$-modular lattices in small dimensions}\label{sec:results}

The key to the computation of secrecy gains is the theta series of the corresponding lattice. We present here two approaches to obtain a closed form expression of the theta series of $2$- and $3$-modular lattices: the modular form approach and the weight enumerator approach. The modular form approach relies on the fact that the theta series of an $\ell$-modular lattice belongs to the space of modular forms generated by some basic functions, which gives a decomposition formula. The formula for even $2$- and $3$-modular lattices is comparatively simple while the formula for $\ell$-modular lattices in general, including the odd lattices, is rather complicated.  A weight enumerator approach is added in the computation for odd $2$-modular lattices in the second subsection. This approach exploits the connection between the weight enumerator of a self-dual code and the theta series of a lattice constructed from this code. But calculating the weight enumerator of the code adds considerable workload.

%*****************************************************************************************%
%A Connection between theta series and weight enumerators                                           %
%*****************************************************************************************%
%\subsection{Connection between theta series and weight enumerators}

%%%%%%%%%%%%%%%%%%%%%%%%%%%%%%%%%%%%%%%%%%%%%%%%%%%%%%%%%
\subsection{Even $2$ and $3$-modular lattices}                                                   %
%%%%%%%%%%%%%%%%%%%%%%%%%%%%%%%%%%%%%%%%%%%%%%%%%%%%%%%%%

%*****************************************************************************************%
%B Invariant theory                                                                                                         %
%*****************************************************************************************%
%\subsection{Theta series as modular forms}
The theta series of modular lattices are modular forms, which, roughly speaking, are functions that stay ``invariant'' under the transformation by certain subgroups of the group SL$_2(\mathbb{Z})$ \cite{modular forms}. The modular form theory shows that theta series as modular forms are expressed in a polynomial in two basic modular forms. We only need a few terms of a theta series to compute the coefficients of this expression and obtain a closed form expression of the theta series. The following lemma plays a crucial role in our calculation of the theta series of $2$- and $3$-modular lattices.
\begin{lemma}\cite{modular lattices}\label{lemma}
The theta series of an even $\ell$-modular lattice of dimension $n=2k$ when $\ell=1,\ 2,\ 3$ belongs to a space of \textit{modular forms} of weight $k$ generated by the functions $\Theta_{2k_0}^{\lambda}(\tau)\Delta_{2k_1}^{\mu}(\tau)$ with integers $\lambda,\ \mu\geq0$ satisfying $k_0\lambda+k_1\mu=k$, where for $\ell=1,\ 2,\ 3$, $k_0=4,\ 2,\ 1$ respectively, $k_1=\frac{24}{1+\ell}$, $\Theta_{2k_0}(\tau)$ denote the theta series of the modular lattices $E_8,\ D_4\mbox{ and }A_2$, respectively, and $
\Delta_{2k_1}(\tau)=\left(\eta(\tau)\eta(\ell \tau)\right)^{k_1}$.
\end{lemma}
%for the modular group 
%$$
%\Gamma_0(\ell)\bigcup \Gamma_0(\ell)t_{\ell},\ t_{\ell}=\left [
%     \begin{array}{cc}
%         0 & \frac{1}{\sqrt{\ell}} \\
%         -\sqrt{\ell} & 0 \\
%     \end{array}
%  \right ]
%$$
%and the character $\chi$
%$$
%\chi(s)=\left(\frac{(-1)^k}{a}\right)\mbox{ for }s=\left [
%     \begin{array}{cc}
%         a & b \\
%         c & d \\
%     \end{array}
%  \right ]\in \Gamma_0(\ell),\ 
%\chi(t_{\ell})=i^k,
%$$
%and for $\ell=1,2\mbox{ and }3$, the algebra of modular forms can be described using Hecke's results:
%\begin{lemma}\cite{modular lattices}\label{lemma}
%
%\end{lemma}
\begin{example} \label{even unimodular}
If $\ell=1$, we read from Lemma \ref{lemma} that $k_0=4$, $k_1=\frac{24}{2}=12$, $\Theta_{2k_0}(\tau)=\Theta_{E_8}(\tau)$ and $\Delta_{2k_1}(\tau)=\eta^{24}(\tau)$. We then deduce that if $\Lambda$ is an even unimodular lattice of dimension $n=2k$ then
\begin{equation}\label{equ:even unimodular}
\Theta_{\Lambda}(\tau)=\sum_{4\lambda+12\mu=k}a_{\mu}\Theta_{E_8}^{\lambda}(\tau)\Delta_{24}^{\mu}(\tau).
\end{equation}
\end{example}

The formula (\ref{equ:even unimodular}) was adopted in \cite{ITW,Gaussian wiretap codes} to compute the secrecy gains of several even unimodular lattices.

In order to write the secrecy function, we need to have the theta series of $\mathbb{Z}^n$ scaled to the right volume.
%Recall the definition of the theta series of a lattice and the Jacobi theta function $\vartheta_3(\tau)$. We have that
%\begin{equation}\label{equ:Zn}
%\left\{
%\begin{array}{ll}
%\Theta_{\mathbb{Z}}(\tau)&=\vartheta_3(\tau),\\
%\Theta_{\mathbb{Z}^k}(\tau)&=\vartheta_3^k(\tau),\\
%\Theta_{k\mathbb{Z}}(\tau)&=\vartheta_3(k^2\tau).\\
%\end{array}
%\right.
%\end{equation}
Now it follows from (\ref{determinant}) that %and (\ref{equ:Zn}) that the theta series is
\begin{equation}\label{theta series of vZ^n}
\Theta_{\ell^{\frac{1}{4}}\mathbb{Z}^n}(\tau)=\vartheta_3^n(\sqrt{\ell}\tau).
\end{equation}

%%%%%%%%%%%%%%%%%%%%%%%%%%%%%%%%%%%%%%
%2-modular
%%%%%%%%%%%%%%%%%%%%%%%%%%%%%%%%%%%%%%

According to Lemma \ref{lemma}, the theta series of an even $2$-modular lattice $\Lambda$ of dimension $n=2k$ can be written as 
\begin{equation}\label{2-modular}
\Theta_{\Lambda}(\tau)=\sum_{2\lambda+8\mu=k}a_{\mu}\Theta_{D_4}^{\lambda}(\tau)\Delta_{16}^{\mu}(\tau),
\end{equation}
where
\begin{equation}\label{D_4}
\begin{array}{ll}
\Theta_{D_4}(\tau)&=\frac{1}{2}\left( \vartheta_3^4(\tau)+\vartheta_4^4(\tau) \right)\\
                             &=1+24q^2+24q^4+96q^6+\cdots\\
\end{array}
\end{equation}
and
$$
\Delta_{16}(\tau)=\left(\eta(\tau)\eta(2\tau)\right)^{8}.
$$
By (\ref{theta and eta}), we can write $\Delta_{16}(\tau)$ in terms of Jacobi theta functions and compute the first few terms:
\begin{equation}\label{Delta_16}
\begin{array}{ll}
\Delta_{16}(\tau)&=\frac{1}{256}\vartheta_2^8(\tau)\vartheta_3^4(\tau)\vartheta_4^4(\tau)\\
                          &=q^2-8q^4+12q^6+\cdots.\\
\end{array}
\end{equation} 
The secrecy function of an even $2$-modular lattice $\Lambda$ of dimension $n$ is then written as
$$
\Xi_{\Lambda}(\tau)=\frac{\vartheta_3^n(\sqrt{2}\tau)}{\sum_{2\lambda+8\mu=k}a_{\mu}\Theta_{D_4}^{\lambda}(\tau)\Delta_{16}^{\mu}(\tau)},
$$
or more conveniently,
$$
\begin{array}{ll}
1/\Xi_{\Lambda}(\tau)&=\sum_{2\lambda+8\mu=k}a_{\mu}\frac{\Theta_{D_4}^{\lambda}(\tau)\Delta_{16}^{\mu}(\tau)}{\vartheta_3^n(\sqrt{2}\tau)}\\
                                   &=\sum_{2\lambda+8\mu=k}a_{\mu}\left(\frac{\Theta_{D_4}(\tau)}{\vartheta_3^4(\sqrt{2}\tau)}\right)^{\lambda}\left(\frac{\Delta_{16}(\tau)}{\vartheta_3^{16}(\sqrt{2}\tau)}\right)^{\mu}.
\end{array}
$$
Now we only need to know the coefficients $a_{\mu}$ in order to compute the weak secrecy gain of a $2$-modular lattice.

Let us compute an example to show how the coefficients $a_\mu$'s in (\ref{2-modular}) are computed. By substituting (\ref{D_4}) and (\ref{Delta_16}) into (\ref{2-modular}), we have a formal sum with coefficients represented by the $a_{\mu}$'s. Then by comparing this formal sum with (\ref{equ:theta series}), we obtain a number of linear equations in the $a_{\mu}$'s. When we have enough equations, the $a_{\mu}$'s can be recovered by solving a linear system. 

\begin{example}
$BW_{16}$ is an even lattice with minimum norm $4$. The theta series of $BW_{16}$ looks like
$$\Theta_{BW_{16}}(\tau)=1+0q^2+A_4q^4+\cdots,\ A_4\neq 0.$$
On the other hand, by (\ref{2-modular}), (\ref{D_4}) and (\ref{Delta_16}),
$$
\begin{array}{ll}
\Theta_{BW_{16}}(\tau)&=a_0\Theta_{D_4}^{4}(\tau)+a_1\Delta_{16}(\tau)\\
                            &=a_0(1+24q^2+\cdots)^4+a_1(q^2+\cdots)\\
                            &=a_0(1+96q^2+\cdots)+a_1(q^2+\cdots)\\
                            &=a_0+(96a_0+a_1)q^2+\cdots.\\
\end{array}
$$
We now have two linear equations in two unknowns $a_0$ and $a_1$
$$
\left \{
\begin{array}{cc}
a_0&=1\\
96a_0+a_1&=0\\
\end{array}
\right.
$$
which gives $a_0=1$ and $a_1=-96$,
yielding the theta series
\begin{equation}
\Theta_{BW_{16}}=\Theta_{D_4}^4-96\Delta_{16}.
\end{equation}
\end{example}
The weak secrecy gain of $BW_{16}$ can then be approximated using Mathematica \cite{Mathematica} (see Fig. \ref{fig:secrecy function of BW_16}):
\begin{equation}
\chi_{BW_{16}}=2.20564.
\end{equation}

Similarly according to Lemma \ref{lemma}, the theta series of an even $3$-modular lattice $\Lambda$ of dimension $n=2k$ can be written as 
\begin{equation}\label{3-modular}
\Theta_{\Lambda}(\tau)=\sum_{\lambda+6\mu=k}a_{\mu}\Theta_{A_2}^{\lambda}(\tau)\Delta_{12}^{\mu}(\tau),
\end{equation}
where
\begin{equation}\label{A_2}
\begin{array}{ll}
\Theta_{A_2}(\tau)&=\vartheta_2(2\tau)\vartheta_2(6\tau)+\vartheta_3(2\tau)\vartheta_3(6\tau)\\
                             &=1+6q^2+0q^4+6q^6+\cdots\\
\end{array}
\end{equation}
and
$$
\Delta_{12}(\tau)=\left(\eta(\tau)\eta(3\tau)\right)^{6}.
$$
We can also compute the first few terms of $\Delta_{12}(\tau)$:
\begin{equation}\label{Delta_12}
\Delta_{12}(\tau)=q^2-6q^4+9q^6+\cdots.
\end{equation} 

The secrecy function of an even $3$-modular lattice $\Lambda$ of dimension $n$ is 
$$
\begin{array}{ll}
1/\Xi_{\Lambda}(\tau)&=\sum_{\lambda+6\mu=k}\frac{a_{\mu}\Theta_{A_2}^{\lambda}(\tau)\Delta_{12}^{\mu}(\tau)}{\vartheta_3^n(\sqrt{3}\tau)}.\\
                                   &=\sum_{\lambda+6\mu=k}a_{\mu}\left(\frac{\Theta_{A_2}(\tau)}{\vartheta_3^2(\sqrt{3}\tau)}\right)^{\lambda}\left(\frac{\Delta_{12}(\tau)}{\vartheta_3^{12}(\sqrt{3}\tau)}\right)^{\mu}.
\end{array}
$$

Table \ref{table:2,3-modular} summarizes the weak secrecy gains of even $2$- and $3$-modular lattices computed. The basic information about these lattices, such as minimum norm and kissing number can be found in \cite{catalogue}.
\begin{table}
\centering
\caption{\label{table:2,3-modular} Weak secrecy gains of the known even $2$- and $3$-modular lattices}
\begin{tabular}{|c|c|c|c|c|}
\hline
dim & lattice     &$\ell$           & theta series                                              &$\chi^w_{\Lambda}$ \\
\hline
\hline
$2$  & $A_{2}$  &$3$ &$\Theta_{A_2}$  &$1.01789$\\
\hline
\hline
$4$  & $D_4$   &$2$ &$\Theta_{D_4}$  &$1.08356$\\
\hline
\hline
$12$  & $K_{12}$  &$3$ &$\Theta_{A_2}^6-36\Delta_{12}$  &$1.66839$\\
\hline
\hline
$14$  & $C^2\times G(2,3)$ &$3$ &$\Theta_{A_2}^7-42\Theta_{A_2}\Delta_{12}$ &$1.85262$\\
\hline
\hline
$16$  & $BW_{16}$  &$2$ &$\Theta_{D_4}^4-96\Delta_{16} $&$2.20564$\\
\hline
\hline
$20$  & $HS_{20}$  &$2$ &$\Theta_{D_4}^5-120\Theta_{D_4}\Delta_{16} $&$3.03551$\\
\hline
\hline
$22$  & $A_2\times A_{11}$ &$3$ &$\Theta_{A_2}^{11}-66\Theta_{A_2}^5\Delta_{12}$ &$3.12527$\\
\hline
$24$  & $L_{24.2}$  &$3$ &$\Theta_{A_2}^{12}-72\Theta_{A_2}^{6}\Delta_{12}$    &$3.92969$\\
&&&$-216\Delta_{12}^2$&\\
\hline
\end{tabular}

\end{table}

%%%%%%%%%%%%%%%%%%%%%%%%%%%%%%%%%%%%%%%%%%%%%%%%%%%%%%%%%
\subsection{Odd $2$-modular lattices}                                                                %
%%%%%%%%%%%%%%%%%%%%%%%%%%%%%%%%%%%%%%%%%%%%%%%%%%%%%%%%%
Odd $2$-modular lattices were constructed in \cite{CBachoc,PSole} via Construction A. They are, by the time of writing this paper, the only known instances of odd $2$-modular lattices. There is a natural connection between the theta series of the lattice constructed from a code $C$ via Construction A and an appropriate weight enumerator of the code $C$. We will exploit this connection to obtain a closed form expression for these lattices.

For the rest of the paper, we will let $K=\mathbb{Q}(\sqrt{-2})$ and $R=\mathfrak{O}_K/3\mathfrak{O}_K$, where the notations are explained in Section \ref{sec:preliminaries}. According to (\ref{eq: ringofintegers}), since $-2\equiv 2\mbox{ mod }4$, the ring of integers $\mathfrak{O}_K$ of $K$ is $\mathfrak{O}_K=\mathbb{Z}[\sqrt{-2}]=\{a+b\sqrt{-2}|a,b\in\mathbb{Z}\}$. Now we consider the decomposition of the prime ideal $3\mathfrak{O}_K$. Since $3=(1+\sqrt{-2})(1-\sqrt{-2})$ and $(1+\sqrt{-2})\mathfrak{O}_K\neq(1-\sqrt{-2})\mathfrak{O}_K$, the ideal $3\mathfrak{O}_K$ splits. According to (\ref{eq: quotientringR}), the quotient ring $R=\mathfrak{O}_K/3\mathfrak{O}_K=\mathbb{F}_3\times \mathbb{F}_3$. Note that the ring $\mathbb{F}_3+v\mathbb{F}_3$ with $v^2=1$ is isomorphic to the ring $\mathbb{F}_3\times \mathbb{F}_3$, through an isomorphism $\delta:\ a(v-1)+b(v+1)\mapsto (a,b)$. We will identify $R=\mathfrak{O}_K/3\mathfrak{O}_K$ with the ring $\mathbb{F}_3+v\mathbb{F}_3$ and use the two notations interchangeably. In particular, we will identify the coset $a+3\mathfrak{O}_K$ with $a\in \mathbb{F}_3$, and the coset $\sqrt{-2}+3\mathfrak{O}_K$ with $v$.  %\textcolor{red}{To be read again!}

Let $C$ be a code of length $n=2k$ over $R=\mathbb{F}_3+v\mathbb{F}_3=\mathfrak{O}_K/3\mathfrak{O}_K$, which is by definition a $R$-submodule of $R^n$. According to Construction A, $\rho^{-1}(C)$ is a lattice over $\mathfrak{O}_K$\footnote{A $k$-dimensional lattice can be defined in a more general setting by a free abelian group of rank $k$.}, say, with generator matrix
$$
\left (
\begin{array}{ccc}
\lambda_{11}&\cdots&\lambda_{1k}\\
&\cdots&\\
\lambda_{k1}&\cdots&\lambda_{kk}\\
\end{array}
\right ).
$$
Let $\frac{1}{\sqrt{3}}\rho^{-1}(C)_{real}$ denote the real lattice defined by the generator matrix
$$
\frac{1}{\sqrt{3}}\left (
\begin{array}{ccccc}
\mbox{Re}(\lambda_{11})&\mbox{Im}(\lambda_{11})&\cdots&\mbox{Re}(\lambda_{1k})&\mbox{Im}(\lambda_{1k})\\
\mbox{Im}(\lambda_{11})&\mbox{Re}(\lambda_{11})&\cdots&\mbox{Im}(\lambda_{1k})&\mbox{Re}(\lambda_{1k})\\
&&\cdots&&\\
\mbox{Re}(\lambda_{k1})&\mbox{Im}(\lambda_{k1})&\cdots&\mbox{Re}(\lambda_{kk})&\mbox{Im}(\lambda_{kk})\\
\mbox{Im}(\lambda_{k1})&\mbox{Re}(\lambda_{k1})&\cdots&\mbox{Im}(\lambda_{kk})&\mbox{Re}(\lambda_{kk})\\
\end{array}
\right ).
$$

Now we look at the theta series of the lattice $\frac{1}{\sqrt{3}}\rho^{-1}(C)_{real}$ constructed from a code $C$ over $R$. 

\begin{definition}\cite{PSole}
The \textit{length function} $l_K$ of an element $r$ in $R=\mathbb{F}_3+v\mathbb{F}_3=\mathfrak{O}_K/3\mathfrak{O}_K$ is defined by
\begin{equation}\label{lengthfunction}
l_K(r)=\inf\{x\bar{x}|x\in r\subset \mathfrak{O}_K\},
\end{equation}
where $\bar{x}$ is the complex conjugation of $x$.
\end{definition}

One computes the length of the nine elements of $R$ as follows:
\begin{equation}\label{length}
\left \{
\begin{array}{ll}
l_K(0)&=0\\
l_K(\pm 1)&=1\\
l_K(\pm v)&=2\\
l_K(\pm 1\pm v)&=3.\\
\end{array}
\right.
\end{equation}
\begin{definition}\cite{PSole}
The \textit{length composition} $n_l(\mathbf{x})$, $l=0,1,2,3$ of a vector $\mathbf{x}$ in $R^n$ counts the number of coordinates of length $l$. The \textit{length weight enumerator} of a code $C$ over $R$ is then defined by
\begin{equation}\label{lengthweightenumerator}
\mbox{lwe}_C(a,b,c,d)=\sum_{\mathbf{c}\in C}a^{n_0(\mathbf{c})}b^{n_1(\mathbf{c})}c^{n_2(\mathbf{c})}d^{n_3(\mathbf{c})}.
\end{equation}
\end{definition}

Define four theta series $\theta_l$, $l=0,1,2,3$ corresponding to the four different lengths of elements of $R$:
\begin{equation}\label{thetatolength}
\left \{
\begin{array}{ll}
\theta_0&=\sum_{x\in 3\mathfrak{O}_K}q^{\frac{x\bar{x}}{3}}\\
\theta_1&=\sum_{x\in 1+3\mathfrak{O}_K}q^{\frac{x\bar{x}}{3}}\\
\theta_2&=\sum_{x\in \sqrt{-2}+3\mathfrak{O}_K}q^{\frac{x\bar{x}}{3}}\\
\theta_3&=\sum_{x\in 1+\sqrt{-2}+3\mathfrak{O}_K}q^{\frac{x\bar{x}}{3}}.\\
\end{array}
\right.
\end{equation} 
%\begin{equation}\label{thetatolength}
%\left \{
%\begin{array}{ll}
%\theta_0&=\sum_{x\in 3\mathfrak{O}_K}q^{\frac{\mbox{Tr}(x\bar{x})}{6}}\\
%\theta_1&=\sum_{x\in 1+3\mathfrak{O}_K}q^{\frac{\mbox{Tr}(x\bar{x})}{6}}\\
%\theta_2&=\sum_{x\in \sqrt{-2}+3\mathfrak{O}_K}q^{\frac{\mbox{Tr}(x\bar{x})}{6}}\\
%\theta_3&=\sum_{x\in 1+\sqrt{-2}+3\mathfrak{O}_K}q^{\frac{\mbox{Tr}(x\bar{x})}{6}},\\
%\end{array}
%\right.
%\end{equation} 
%where the \textit{trace function} is defined by 
%$$
%\mbox{Tr}(z)=z+\bar{z},\ z\in K.
%$$
Recalling that $\mathfrak{O}_K=\{a+b\sqrt{-2}|a,b\in\mathbb{Z}\}$, the theta series are written as double sums.
\begin{equation}\label{infinitesum}
\left \{
\begin{array}{ll}
\theta_0&=\sum_{a\in\mathbb{Z}}\sum_{b\in\mathbb{Z}}q^{3a^2+6b^2}\\
\theta_1&=\sum_{a\in\mathbb{Z}}\sum_{b\in\mathbb{Z}}q^{3(a+\frac{1}{3})^2+6b^2}\\
\theta_2&=\sum_{a\in\mathbb{Z}}\sum_{b\in\mathbb{Z}}q^{3a^2+6(b+\frac{1}{3})^2}\\
\theta_3&=\sum_{a\in\mathbb{Z}}\sum_{b\in\mathbb{Z}}q^{3(a+\frac{1}{3})^2+6(b+\frac{1}{3})^2}.\\
\end{array}
\right.
\end{equation} 

%We then want to investigate these four types of infinite sums:
%\begin{equation}\label{infinitesum}
%\left \{
%\begin{array}{ll}
%N(0,0)&=3m^2+6s^2\\
%N(1,0)&=3(m+\frac{1}{3})^2+6s^2\\
%N(0,1)&=3m^2+6(s+\frac{1}{3})^2\\
%N(1,1)&=3(m+\frac{1}{3})^2+6(s+\frac{1}{3})^2.\\
%\end{array}
%\right.
%\end{equation} 
We already know how to handle the $lm^2$ type of infinite sum, namely, 
$$
\sum_{m\in \mathbb{Z}}q^{lm^2}=\sum_{m\in \mathbb{Z}}(q^l)^{m^2}=\vartheta_3(l\tau).
$$
For the $(3m+1)^2$ type of infinite sum, we first observe that, on one hand,
$$
\sum_{m\in\mathbb{Z}}q^{m^2}=\sum_{m\in\mathbb{Z}}q^{(3m)^2}+\sum_{m\in\mathbb{Z}}q^{(3m+1)^2}+\sum_{m\in\mathbb{Z}}q^{(3m-1)^2}
$$
and, on the other hand,
$$
\sum_{m\in\mathbb{Z}}q^{(3m+1)^2}=\sum_{m\in\mathbb{Z}}q^{(3m-1)^2}.
$$
We then conclude that
$$
\begin{array}{ll}
\sum_{m\in\mathbb{Z}}q^{(3m+1)^2}&=\frac{1}{2}\left(  \sum_{m\in\mathbb{Z}}q^{m^2} -   \sum_{m\in\mathbb{Z}}q^{(3m)^2}\right)\\
                                                           &=\frac{1}{2}\left(\vartheta_3(\tau)-\vartheta_3(9\tau)\right).\\
\end{array}
$$
The four theta series defined above are then computed as
\begin{equation}\label{thetatolength}
\left \{
\begin{array}{ll}
\theta_0&=\vartheta_3(3\tau)\vartheta_3(6\tau)\\
\theta_1&=\frac{1}{2}\left(\vartheta_3(\frac{\tau}{3})-\vartheta_3(3\tau)\right)\vartheta_3(6\tau)\\
\theta_2&=\frac{1}{2}\vartheta_3(3\tau)\left(\vartheta_3(\frac{2\tau}{3})-\vartheta_3(6\tau)\right)\\
\theta_3&=\frac{1}{4}\left(\vartheta_3(\frac{\tau}{3})-\vartheta_3(3\tau)\right)\left(\vartheta_3(\frac{2\tau}{3})-\vartheta_3(6\tau)\right).\\
\end{array}
\right.
\end{equation} 

\begin{theorem}
\begin{equation}\label{Macwillim} 
\Theta_{\frac{1}{\sqrt{3}}\rho^{-1}(C)}(q)=\mbox{lwe}_C(\theta_0,\theta_1,\theta_2,\theta_3).
\end{equation}
\end{theorem}
\begin{proof}The theta series of the lattice $\frac{1}{\sqrt{3}}\rho^{-1}(C)$ is by definition
$$
\begin{array}{ll}
\Theta_{\frac{1}{\sqrt{3}}\rho^{-1}(C)}(\tau)&=\sum_{\mathbf{\lambda}\in\frac{1}{\sqrt{3}}\rho^{-1}(C)}q^{||\mathbf{\lambda}||^2}\\
                                                               &=\sum_{\mathbf{c}\in C}\sum_{\mathbf{x}\in\frac{1}{\sqrt{3}}(\mathbf{c}+3\mathfrak{O}_K^k)}q^{\mathbf{x}\bar{\mathbf{x}}}\\
                                                              &=\sum_{\mathbf{c}\in C}\theta_0^{n_0(\mathbf{c})}\theta_1^{n_1(\mathbf{c})}\theta_2^{n_2(\mathbf{c})}\theta_3^{n_3(\mathbf{c})}\\
                                                              &=\mbox{lwe}_C(\theta_0,\theta_1,\theta_2,\theta_3).\\
\end{array}
$$
\end{proof}

As it was remarked in \cite{CBachoc} (Remark 3.8) and later proved in \cite{PSole}, if $C$ is a self-dual code over $R$ with respect to Hermitian inner product, then $\frac{1}{\sqrt{3}}\rho^{-1}(C)_{real}$ is an odd $2$-modular lattice.
\begin{example}\label{8-dimensional} 
A Hermitian self-dual code $C$ over $R$ of length $4$ was constructed in \cite{PSole}. It is a linear code with a generator matrix
\begin{equation}\label{generatormatrix}
G^H=\left[
\begin{array}{cccc}
1&0&v&-1-v\\
0&1&-1+v&v\\
\end{array}
\right].
\end{equation}
One can generate all the $81$ codewords and compute the length weight enumerator:
$$
\begin{array}{ll}
\mbox{lwe}_C(a,b,c,d)&=a^4+4a^2d^2+16abcd+8ad^3+8b^3d\\
                                   &\ +4b^2c^2+24bcd^2+8c^3d+8d^4.\\
\end{array}
$$
The theta series of the $8$-dimensional odd $2$-modular lattice is then computed by (\ref{Macwillim}) (using a computer software, for example, Mathematica \cite{Mathematica} to output the first few terms).
$$
\begin{array}{l}
\Theta_{\frac{1}{\sqrt{3}}\rho^{-1}(C)}(\tau)\\
                                =\vartheta_3(3\tau)^4\vartheta_3(6\tau)^4\\
                                \ +\frac{1}{4}\vartheta_3(3\tau)^3\left(\vartheta_3(\frac{\tau}{3})-\vartheta_3(3\tau)\right)\left(\vartheta_3(\frac{2\tau}{3})-\vartheta_3(6\tau)\right)^4\\                               
                                \ +\frac{3}{2}\vartheta_3(3\tau)^2\vartheta_3(6\tau)^2\left(\vartheta_3(\frac{\tau}{3})-\vartheta_3(3\tau)\right)^2\left(\vartheta_3(\frac{2\tau}{3})-\vartheta_3(6\tau)\right)^2\\
                                \ +\frac{5}{8}\vartheta_3(3\tau)\vartheta_3(6\tau)\left(\vartheta_3(\frac{\tau}{3})-\vartheta_3(3\tau)\right)^3\left(\vartheta_3(\frac{2\tau}{3})-\vartheta_3(6\tau)\right)^3\\
                                \ +\frac{1}{4}\vartheta_3(6\tau)^3\left(\vartheta_3(\frac{\tau}{3})-\vartheta_3(3\tau)\right)^4\left(\vartheta_3(\frac{2\tau}{3})-\vartheta_3(6\tau)\right)\\

                                \ +\frac{1}{32}\left(\vartheta_3(\frac{\tau}{3})-\vartheta_3(3\tau)\right)^4\left(\vartheta_3(\frac{2\tau}{3})-\vartheta_3(6\tau)\right)^4\\
                                =1+32q^2+128q^3+240q^4+\cdots.\\
\end{array}
$$
%(Rescaling...the trace was considered divided by 2 by Sole and there was another divided by 3 by Bachoc omitted)
\end{example}

This method has the advantage of being self-contained in its deduction. But the computation of the weight enumerator of the code $C$ is tedious and, worse still, as the dimension increases, it may become infeasible. Let us fall back to the first approach adopted in the previous subsection.

First we need a formula similar to Lemma \ref{lemma} which deals with the theta series of odd $2$-modular lattices. There is indeed a formula which deals with the theta series of $\ell$-modular lattice, including the odd lattices, for $\ell=1,2,3,5,6,7,11,14,15,23$ discovered by E. M. Rains and N. J. A. Sloane.

\begin{lemma}\cite{N.J.A Sloane}\label{lemma2}
Define
$$
f_1(\tau)=\Theta_{C^{\ell}}(\tau),
$$
where the lattice $C^{\ell}$ is as defined in (\ref{Z+lZ}).
Let $\eta^{\ell}(\tau)=\Pi_{d|\ell}\eta(d\tau)$ and let $D_{\ell}=24,16,12,8,8,6,4,4,4,2$ corresponding to $\ell=1,2,3,5,6,7,11,14,15,23$. Define
$$
f_2(\tau)=\left\{
\begin{array}{ll}
\left (\frac{\eta^{\ell}(\frac{\tau}{2})\eta^{\ell}(2\tau)}{\eta^{\ell}(\tau)^2} \right)^{\frac{D_{\ell}}{\mbox{dim }C^{\ell}}},&\ell\mbox{ is odd;}\\
\left (\frac{\eta^{(\frac{\ell}{2})}(\frac{\tau}{2})\eta^{(\frac{\ell}{2})}(4\tau)}{\eta^{(\frac{\ell}{2})}(\tau)\eta^{(\frac{\ell}{2})}(2\tau)} \right)^{\frac{D_{\ell}}{\mbox{dim }C^{\ell}}}   ,&\ell\mbox{ is even.}\\
\end{array}
\right.
$$
The theta series of an $\ell$-modular lattice $\Lambda$ of dimension $k\mbox{dim}(C^{(\ell)})$ can be written as
\begin{equation}\label{l-modular decomposition}
\Theta_{\Lambda}(\tau)=f_1(\tau)^k\sum_{i=0}^{\lfloor k\mbox{ ord}_1(f_1)\rfloor}a_if_2(\tau)^i,
\end{equation}
where $\mbox{ord}_1(f_1)$ is the \textit{divisor} of the modular form $f_1(\tau)$, which, in this case, is $\frac{1}{8}\sum_{d|\ell}d$ if $\ell$ is odd and $\frac{1}{6}\sum_{d|\ell}d$ if $\ell$ is even.
\end{lemma}

Let us now take $\ell=2$. Then $C^2=\mathbb{Z}\oplus\sqrt{2}\mathbb{Z}$ hence
\begin{equation}\label{2modularoddf1}
\begin{array}{ll}
f_1(\tau)&=\Theta_{C^2}(\tau)\\
              &=\vartheta_3(\tau)\vartheta_3(2\tau)\\
              &=1+2q+2q^2+4q^3+\cdots.\\
\end{array}
\end{equation}
Next, $\mbox{ord}_1(f_1)$ is computed to be $\frac{1}{2}$. $D_2=16$. Finally since $2$ is even
$$
f_2(\tau)=\left (\frac{\eta(\frac{\tau}{2})\eta(4\tau)}{\eta(\tau)\eta(2\tau)} \right)^{\frac{16}{2}}=\frac{\vartheta_2^2(2\tau)\vartheta_4^2(\tau)}{4\vartheta_3^2(\tau)\vartheta_3^2(2\tau)}.
$$
We observe that the denominator of $f_2(\tau)$ is  $4f_1^2(\tau)$. We then define a function
\begin{equation}\label{2modularoddDelta4}
\begin{array}{ll}
\Delta_4(\tau)&\triangleq f_1^2(\tau)f_2(\tau)\\
                       &=\frac{1}{4}\vartheta_2^2(2\tau)\vartheta_4^2(\tau)\\
                       &=q-4q^2+4q^3+\cdots,\\
\end{array}
\end{equation}
and rewrite (\ref{l-modular decomposition}) in the form of (\ref{2-modular}):
\begin{equation}\label{2-modularodd}
\Theta_{\Lambda}(\tau)=\sum_{i=0}^{\lfloor\frac{k}{2}\rfloor}a_if_1^{k-2i}(\tau)\Delta_4^i(\tau).
\end{equation}

For lattices in small dimensions, the first few terms of the theta series can be computed numerically using computer softwares, for example, Magma \cite{Magma}.
\begin{example}\label{grammatrix}
A generator matrix of the $8$-dimensional odd $2$-modular lattice in Example \ref{8-dimensional} can be computed from the generator matrix (\ref{generatormatrix}) of the code $C$:
$$
M=\frac{1}{\sqrt{3}}\left[
\begin{array}{cccccccc}
1&0&0&0&0&\sqrt{2}&-1&-\sqrt{2}\\
0&\sqrt{2}&0&0&1&0&-1&-\sqrt{2}\\
0&0&1&0&-1&\sqrt{2}&0&\sqrt{2}\\
0&0&0&\sqrt{2}&1&-\sqrt{2}&1&0\\
0&0&0&0&3&0&0&0\\
0&0&0&0&0&3\sqrt{2}&0&0\\
0&0&0&0&0&0&3&0\\
0&0&0&0&0&0&0&3\sqrt{2}\\
\end{array}
\right].
$$
To make the typing easy, we compute the Gram matrix 
$$
MM^{T}=\left[
\begin{array}{cccccccc}
2&1&0&-1&0&2&-1&-2\\
1&2&-1&0&1&0&-1&-2\\
0&-1&2&-1&-1&2&0&2\\
-1&0&-1&2&1&-2&1&0\\
0&1&-1&1&3&0&0&0\\
2&0&2&-2&0&6&0&0\\
-1&-1&0&1&0&0&3&0\\
-2&-2&2&0&0&0&0&6\\
\end{array}
\right]
$$ 
and input it to Magma to generate the lattice $\Lambda$. The first few terms of $\Theta_{\frac{1}{\sqrt{3}}\rho^{-1}(C)}(\tau)$ can be obtained (by the command ThetaSeries($\Lambda$,0,4);):
$$
\Theta_{\frac{1}{\sqrt{3}}\rho^{-1}(C)}(q)=1+32q^2+128q^3+240q^4+\cdots.
$$
Now in dimension $8$, the theta series of a $2$-modular lattice can be written as 
\begin{equation}\label{2-modular8-dimensional}
\begin{array}{l}
a_0f_1(\tau)^4+a_1f_1(\tau)^2\Delta_4(\tau)+a_2\Delta_4(\tau)^2\\
=a_0(1+8q+32q^2+\cdots)+a_1(q+0q^2+\cdots)\\
\ +a_2(q^2+\cdots)\\
=a_0+(8a_0+a_1)q+(32a_0+0+a_2)q^2+\cdots.\\
\end{array}
\end{equation}
We then have three linear equations in three unknowns $a_0$, $a_1$ and $a_2$
$$
\left \{
\begin{array}{cc}
a_0&=1\\
8a_0+a_1&=0\\
32a_0+a_2&=32,\\
\end{array}
\right.
$$
which gives $a_0=1$, $a_1=-8$ and $a_2=0$,
yielding the theta series
\begin{equation}
\Theta_{\frac{1}{\sqrt{3}}\rho^{-1}(C)}(q)=f_1(\tau)^4-8f_1(\tau)^2\Delta_4(\tau).
\end{equation}
\end{example}

Theta series of the twelve odd $2$-modular lattices constructed in \cite{PSole} are computed and shown in Table \ref{table:2-modular odd}, as polynomials in $f_1$ and $\Delta_4$ for simplicity. Their weak secrecy gains are approximated using Mathematica \cite{Mathematica}. 
%$1.22672$$1.49049$$2.06968$$2.35656$$2.70165$$3.11161$$3.60867$$4.21349$$4.98013$
%$5.72703$
\begin{table}
\centering
\caption{\label{table:2-modular odd} Weak secrecy gains of odd $2$-modular lattices constructed from self-dual codes}
\begin{tabular}{|c|c|c|}
\hline
dim        & theta series                                              &$\chi^w_{\Lambda}$ \\
\hline
\hline
$8$        &$f_1^4-8f_1^2\Delta_4$  &$1.22672$\\
\hline
\hline
$12$       &$f_1^6-12f_1^4\Delta_4$  &$1.49049$\\
\hline
\hline
$16$       &$f_1^8-16f_1^6\Delta_4$  &$2.06968$\\
\hline
\hline
$18$       &$f_1^9-18f_1^7\Delta_4+18f_1^5\Delta_4^2$ &$2.35656$\\
\hline
\hline
$20$       &$f_1^{10}-20f_1^8\Delta_4+40f_1^6\Delta_4^2$&$2.70165$\\
\hline
\hline
$22$       &$f_1^{11}-22f_1^9\Delta_4+66f_1^7\Delta_4^2-4f_1^5\Delta_4^3$&$3.11161$\\
\hline
\hline
$24$       &$f_1^{12}-24f_1^{10}\Delta_4+96f_1^8\Delta_4^2-28f_1^6\Delta_4^3$ &$3.60867$\\
\hline
\hline
$26$       &$f_1^{13}-26f_1^{11}\Delta_4+130f_1^9\Delta_4^2+-80f_1^7\Delta_4^3$ &$4.21349$\\
\hline
\hline
$28$   &$f_1^{14}-28f_1^{12}\Delta_4+168f_1^{10}\Delta_4^2$ &$4.98013$\\
&$-176f_1^8\Delta_4^3+32f_1^6\Delta_4^4$&\\
\hline
\hline
$30$       &$f_1^{15}-30f_1^{13}\Delta_4+210f_1^{11}\Delta_4^2$ &$5.72703$\\
&$-282f_1^9\Delta_4^3+112f_1^7\Delta_4^4$&\\
\hline
\end{tabular}

\end{table}
%%%%%%%%%%%%%%%%%%%%%%%%%%%%%%%%%%%%%%%%%%%%%%%%%%%%%%%%%%%%%%%%
%%%%%%%%%%%%%%%%%%%%%%%%%%%%%%%%%%%%%%%%%%%%%%%%%%%%%%%%%%%%%%%%
\section{Best known lattices}       \label{sec:comparision}                                                %%%
%%%%%%%%%%%%%%%%%%%%%%%%%%%%%%%%%%%%%%%%%%%%%%%%%%%%%%%%%%%%%%%%
%%%%%%%%%%%%%%%%%%%%%%%%%%%%%%%%%%%%%%%%%%%%%%%%%%%%%%%%%%%%%%%%
Now that we have computed the weak secrecy gains of several $2$- and $3$-modular lattices, we want to compare them with the best unimodular lattices in their respective dimensions. Figure \ref{fig: 2n3vs1} compares the secrecy gains of the best unimodular lattices with the weak secrecy gains of the $2$- and $3$-modular lattices we have computed. We can see that most of these even $2$- and $3$-modular lattices, indicated by disconnected big dots, outperform the unimodular lattices except in dimension $22$, and three of the odd $2$-modular lattices, indicated by disconnected small dots, outperform the unimodular lattices, in particular, in dimension $18$, the odd $2$-modular lattice has the best secrecy gain known by now. 

\begin{figure}[htp]

\centering
\includegraphics[width=80mm, height=60mm]{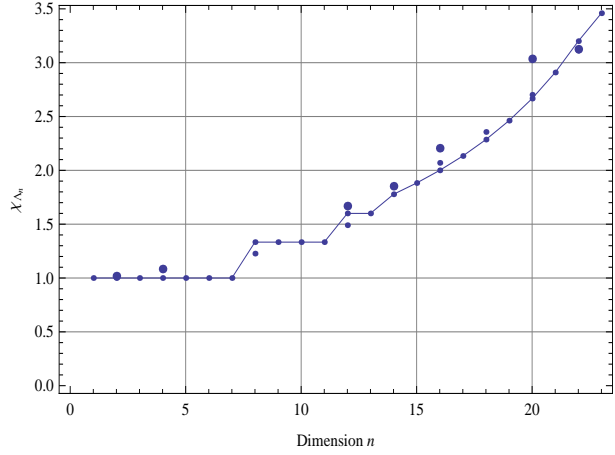}
\caption{\label{fig: 2n3vs1}The weak secrecy gain of $2$- and $3$-modular lattices vs. unimodular lattices as a function of the dimension $n$}
\end{figure}

Table \ref{table:1,2,3-modular} gives a list of $2$- and $3$-modular lattices out-performing the best unimodular lattices.

\begin{table}
\centering
\caption{\label{table:1,2,3-modular}List of $2$- and $3$-modular lattices out-performing the best unimodular lattices}
\begin{tabular}{|c|c|c|c|}
\hline
dim & lattice                          & $\ell$                                              &$\chi_{\Lambda}$ \\
\hline
\hline
$2$  & $\mathbb{Z}^2$         &$1$  &$1$\\
\hline
$2$  & $A_{2}$                       &$3$  &$\geq1.01789$\\
\hline
\hline
$4$  & $\mathbb{Z}^4$         &$1$  &$1$\\
\hline
$4$  & $D_4$                         &$2$  &$\geq1.08356$\\
\hline
\hline
$12$  & $D_{12}^+$              &$1$  &$1.6$\\
\hline
$12$  & $K_{12}$                   &$3$  &$\geq1.66839$\\
\hline
\hline
$14$  & $(E_7^2)^+$             &$1$ &$1.77778$\\
\hline
$14$  & $C_2\times G(2,3)$  &$3$ &$\geq1.85262$\\
\hline
\hline
$16$  & $(D_{8}^2)^+$            &$1$&$2$\\
\hline
$16$  & $\frac{1}{\sqrt{3}}\rho^{-1}(C)_{real}$                &$2$&$\geq 2.06968$\\
\hline
$16$  & $BW_{16}$                &$2$&$\geq2.20564$\\
\hline
\hline
$18$  & $(D_{6}^3)^+$ or $(A_9^2)^+$          &$1$&$2.28571$\\
\hline
$18$  & $\frac{1}{\sqrt{3}}\rho^{-1}(C)_{real}$                &$2$&$\geq2.35656$\\
\hline
\hline
$20$  & $(A_{5}^4)^+$          &$1$&$2.66667$\\
\hline
$20$  & $\frac{1}{\sqrt{3}}\rho^{-1}(C)_{real}$                &$2$&$\geq 2.70165$\\
\hline
$20$  & $HS_{20}$                &$2$&$\geq3.03551$\\
\hline
\hline
$22$  & $(A_1^{22})^+$        &$1$ &$3.2$\\
\hline
$22$  & $A_2\times A_{11}$ &$3$ &$\geq3.12527$\\
\hline
\end{tabular}

\end{table}

%*************************************************************************************%
%                                                                                                                                %
%SECTION V CONCLUSION AND FUTURE WORK                                                         %
%                                                                                                                                %
%*************************************************************************************%
\section{Conclusion and future work}\label{sec:conclu}

This paper computes the weak secrecy gains of several known $2$- and $3$-modular lattices in small dimensions. Most of the even $2$- and $3$-modular lattices and three of the odd $2$-modular lattices have a higher secrecy gain than the best unimodular lattices. We then conclude that, at least in dimensions up to $23$, $2$- and $3$-modular lattices are better option for Gaussian wiretap channel.

A line of future work would naturally be investigating $\ell$-modular lattices for other values of $\ell$ to understand if bigger $\ell$ allows better modular lattices in terms of secrecy gain. Also, more $2$- and $3$-modular lattice examples should be found to get a better understanding of why they have a higher secrecy gain, since a classification of such lattices is currently unavailable even in small dimensions.

\section*{Acknowledgment}
The research of F. Lin and of F. Oggier for this work is supported by the Singapore National Research
Foundation under the Research Grant NRF-RF2009-07. The research of P. Sol\'e for this work is supported by Merlion project 1.02.10.

The authors would like to thank Christine Bachoc for helpful discussions.

\end{document}